\documentclass{amsart}

\usepackage[T1]{fontenc}
\usepackage{amssymb}
\usepackage{enumitem}
\usepackage{mathrsfs}
\usepackage[colorlinks, linkcolor=blue, citecolor=blue, urlcolor=blue]{hyperref}
\usepackage{geometry}
\usepackage{tikz}
\newtheorem{theorem}{Theorem}[section]
\newtheorem{corollary}[theorem]{Corollary}
\newtheorem{lemma}[theorem]{Lemma}

\theoremstyle{definition}
\newtheorem{definition}[theorem]{Definition}
 
\newtheorem{fact}[theorem]{Fact}

\newcommand{\fra}{\mathfrak{a}}
\newcommand{\veq}{\rotatebox[origin=c]{-90}{$ =\ \ $}}

\DeclareMathOperator{\dom}{dom}
\DeclareMathOperator{\ran}{ran}
\DeclareMathOperator{\fin}{fin}
\DeclareMathOperator{\seq}{seq^{1-1}}

\title[The finitary partitions with $n$ non-singleton blocks]{The finitary partitions with $n$ non-singleton blocks\\ of a set}

\author{Yifan Hu}
\address{School of Philosophy\\
Wuhan University\\
No.~299 Bayi Road\\
Wuhan 430072\\
Hubei Province\\
People's Republic of China}
\email{kdygqh@qq.com}

\author{Guozhen Shen}
\address{School of Philosophy\\
Wuhan University\\
No.~299 Bayi Road\\
Wuhan 430072\\
Hubei Province\\
People's Republic of China}
\email{shen\_guozhen@outlook.com}

\thanks{Shen was partially supported by NSFC No.~12101466.}



\begin{document}

\begin{abstract}
A partition is finitary if all its blocks are finite. For a cardinal $\mathfrak{a}$ and a natural number $n$, let $\mathrm{fin}(\mathfrak{a})$ and  $\mathscr{B}_{n}(\mathfrak{a})$ be the cardinalities of the set of finite subsets and the set of finitary partitions with exactly $n$ non-singleton blocks of a set which is of cardinality~$\mathfrak{a}$, respectively. In this paper, we prove in $\mathsf{ZF}$ (without the axiom of choice) that for all infinite cardinals $\mathfrak{a}$ and all non-zero natural numbers $n$, 
\[
(2^{\mathscr{B}_{n}(\mathfrak{a})})^{\aleph_0}=2^{\mathscr{B}_{n}(\mathfrak{a})}
\]
and 
\[
2^{\mathrm{fin}(\mathfrak{a})^n}=2^{\mathscr{B}_{2^n-1}(\mathfrak{a})}.
\]
It is also proved consistent with $\mathsf{ZF}$ that there exists an infinite cardinal $\mathfrak{a}$ such that 
\[
2^{\mathscr{B}_{1}(\mathfrak{a})}<2^{\mathscr{B}_{2}(\mathfrak{a})}<2^{\mathscr{B}_{3}(\mathfrak{a})}<\cdots<2^{\mathrm{fin}(\mathrm{fin}(\mathfrak{a}))}.
\]
\end{abstract}

\maketitle

\section{Introduction}
A partition is finitary if all its blocks are finite. For a cardinal $\mathfrak{a}$ and a natural number $n$, let $\mathrm{fin}(\mathfrak{a})$,  $\mathscr{B}_{\mathrm{fin}}(\mathfrak{a})$, and $\mathscr{B}_{n}(\mathfrak{a})$ be the cardinalities of the set of finite subsets, the set of finitary partitions with finitely many non-singleton blocks, and the set of finitary partitions with exactly $n$ non-singleton blocks of a set which is of cardinality~$\mathfrak{a}$, respectively. For cardinals $\fra$ and $\mathfrak{b}$, we write  $\fra\leqslant^\ast\mathfrak{b}$ if there are sets $A$ and $B$ of cardinalities $\fra$ and $\mathfrak{b}$, respectively, such that there is a  surjection from a subset of $B$ onto $A$.

In 1961, using Ramsey's theorem and a subtle coding technique, L\"auchli~\cite{Lauchli1961} proved in $\mathsf{ZF}$ (i.e., the Zermelo--Fraenkel set theory without the axiom of choice) that for all infinite cardinals $\mathfrak{a}$,
\[
(2^{\fin(\mathfrak{a})})^{\aleph_0}=2^{\fin(\mathfrak{a})},
\]
which is called L\"auchli's lemma in \cite{Halbeisen}.
As a consequence,  $2^{2^{\mathfrak{a}}}+2^{2^{\mathfrak{a}}}=2^{2^{\mathfrak{a}}}$. Recently, Sonpanow and Vejjajiva~\cite{sv2024} proved many analogs of this result with $2^{\mathfrak{a}}$ replaced by some other cardinals. In \cite{Shen2021}, the second author proved that for all infinite cardinals $\fra$ and all natural numbers $n$,
\begin{equation}\label{ss01}
	2^{\fin(\mathfrak{a})^n}=2^{[\fin(\mathfrak{a})]^n},
\end{equation}
whereas the existence of an infinite cardinal $\mathfrak{a}$ such that 
\[
2^{\fin(\mathfrak{a})}<2^{\fin(\mathfrak{a})^2}<2^{\fin(\mathfrak{a})^3}<\cdots<2^{\fin(\fin(\mathfrak{a}))}
\]
is consistent with $\mathsf{ZF}$.

In this paper, we improve the above results by showing that 
for all infinite cardinals $\mathfrak{a}$ and all non-zero natural numbers $n$, 
\begin{equation}\label{hs01}
	(2^{\mathscr{B}_{n}(\mathfrak{a})})^{\aleph_0}=2^{\mathscr{B}_{n}(\mathfrak{a})}
\end{equation}
and 
\begin{equation}\label{hs02}
	2^{\mathrm{fin}(\fra)^n}=2^{\mathscr{B}_{2^n-1}(\mathfrak{a})},
\end{equation}
whereas the existence of an infinite cardinal $\fra$ such that 
\[
2^{\mathscr{B}_{1}(\mathfrak{a})}<2^{\mathscr{B}_{2}(\mathfrak{a})}<2^{\mathscr{B}_{3}(\mathfrak{a})}<\cdots<2^{\mathrm{fin}(\mathrm{fin}(\mathfrak{a}))}
\]
is consistent with $\mathsf{ZF}$. 
The cardinal $\mathscr{B}_{\fin}(\fra)$ is introduced in~\cite{Shen2024}, where it is proved that $2^\fra\not=\mathscr{B}_{\fin}(\fra)$ for all infinite cardinals $\fra$. As a corollary of \eqref{hs02}, we obtain that 
\[
2^{\fin(\fin(\fra))}=2^{\mathscr{B}_{\fin}(\fra)},
\]
though the stronger statement $\fin(\fin(\fra))\leqslant^\ast\mathscr{B}_{\fin}(\fra)$ cannot be proved in $\mathsf{ZF}$. 

The main ideas of the proofs of \eqref{hs01} and \eqref{hs02} are as follows. First, we define a new cardinal $\mathscr{O}_n(\fra)$ related to $\fra$, which is the cardinality of the set of all ordered $n$-tuples of pairwise disjoint finite subsets of a set which is of cardinality~$\fra$. Then clearly $\mathscr{B}_{n}(\fra)\leqslant^\ast\mathscr{O}_n(\fra)$, and by combining the ideas in \cite{Lauchli1961} and \cite{Shen2021}, we prove that 
\begin{equation}\label{hs03}
	(2^{{\mathscr{O}_{n}(\mathfrak{a})}})^{\aleph_0}=2^{\mathscr{B}_{n}(\mathfrak{a})},
\end{equation}
which establishes \eqref{hs01}. Second, we show that 
\[
{\mathrm{fin}(\fra)^n}={\mathscr{O}_{2^n-1}(\mathfrak{a})},
\]
which, along with \eqref{hs03}, establishes \eqref{hs02}.

\section{Basic notions and facts }
We briefly indicate  our use of some terminology and notations. The cardinality of a set $A$ is denoted by $|A|$. We use lowercase Fraktur letters $\mathfrak{a},\mathfrak{b}$ for cardinals. 
For a function $f$, we denote the domain of $f$ by $\dom(f)$, the range of $f$ by $\ran(f)$, the image of $A$ under $f$ by $f[A]$, the inverse image of $A$ under $f$ by $f^{-1}[A]$, and the restriction of $f$ to $A$ by $f{\upharpoonright} A$. For functions $f$ and $g$, we use $g\circ f$ for the composition of $g$ and $f$.

For a set $A$, a partition of $A$ is a family $P$ of pairwise disjoint non-empty subsets of $A$ such that $\bigcup P=A$. Elements of partitions are also called blocks. We say that a partition $P$ is finitary if all blocks of $P$ are finite, and write ns$(P)$ for the set of non-singleton blocks of $P$.

\begin{definition}
	Let $A,B$ be arbitrary sets, let $\mathfrak{a}=|A|$, and let $\mathfrak{b}=|B|$.
	\begin{enumerate}
		
		\item $\mathfrak{a}\leqslant\mathfrak{b}$ means that there is an injection from $A$ into $B$. 
		\item $\mathfrak{a}\leqslant^{\ast}\mathfrak{b}$ means that there is a surjection from a subset of $B$ onto $A$.
		\item $\mathfrak{a}\nleqslant\mathfrak{b}$ ($\mathfrak{a}\nleqslant^{\ast}\mathfrak{b}$) denotes the negation of $\mathfrak{a}\leqslant\mathfrak{b}$ ($\mathfrak{a}\leqslant^{\ast}\mathfrak{b}$).
		\item $\mathfrak{a}<\mathfrak{b}$ means that $\mathfrak{a}\leqslant\mathfrak{b}$ and $\mathfrak{b}\nleqslant\mathfrak{a}$.
		\item $\mathfrak{a}=^{\ast}\mathfrak{b}$ means that $\mathfrak{a}\leqslant^{\ast}\mathfrak{b}$ and $\mathfrak{b}\leqslant^{\ast}\mathfrak{a}$.
	\end{enumerate}
\end{definition}

The Cantor--Bernstein theorem states that if $\mathfrak{a}\leqslant\mathfrak{b}$ and $\mathfrak{b}\leqslant\mathfrak{a}$ then $\mathfrak{a}=\mathfrak{b}$. Clearly, if $\mathfrak{a}\leqslant\mathfrak{b}$, then $\mathfrak{a}\leqslant^{\ast}\mathfrak{b}$, and if $\mathfrak{a}\leqslant^{\ast}\mathfrak{b}$, then $2^{\mathfrak{a}}\leqslant2^{\mathfrak{b}}$.

\begin{definition}
	Let $n,m_1,\dots,m_n\in\omega$. Let $A,B$ be arbitrary sets, let $\mathfrak{a}=|A|$, and let $\mathfrak{b}=|B|$.
	\begin{enumerate}	
		\item $A^{B}$ is the set of all injections from $B$ into $A$; $\mathfrak{a}^{\mathfrak{b}}=|A^{B}|$.
		
		\item $\mathrm{seq}(A)=\bigcup_{k\in\omega}A^{k}$; $\mathrm{seq}(\mathfrak{a})=|\mathrm{seq}(A)|$.
		\item $A^{\underline{B}}$ is the set of all injections from $B$ into $A$; $\mathfrak{a}^{\underline{\mathfrak{b}}}=|A^{\underline{B}}|$.
		
		\item $\seq(A)=\bigcup_{k\in\omega}A^{\underline{k}}$; $\seq(\mathfrak{a})=|\seq(A)|$.
		\item $[A]^B=\{C\subseteq A\mid |C|=|B|\}$; $[\mathfrak{a}]^{\mathfrak{b}}=|[A]^B|$.
		
		\item $\fin(A)=\bigcup_{k\in\omega}[A]^k$; $\fin(\mathfrak{a})=|\fin(A)|$.

		\item $\mathscr{B}_n(A)$ is the set of all finitary partitions of $A$ with exactly $n$ non-singleton blocks;  $\mathscr{B}_{n}(\mathfrak{a})=|\mathscr{B}_n(A)|$.
		
		\item $\mathscr{B}_{\fin}(A)=\bigcup_{k\in\omega}\mathscr{B}_k(A)$; $\mathscr{B}_{\fin}(\mathfrak{a})=|\mathscr{B}_{\fin}(A)|$.
		\item $\mathscr{O}_{m_1,\dots,m_n}(A)$ is the set  of all ordered $n$-tuples of pairwise disjoint  subsets of $A$ which are of cardinalities $m_1,\cdots,m_n$, respectively; $\mathscr{O}_{m_1,\dots,m_n}(\mathfrak{a})=|\mathscr{O}_{m_1,\dots,m_n}(A)|$.
		\item $\mathscr{O}_{n}(A)=\bigcup_{k_1,\dots,k_n\in\omega}\mathscr{O}_{k_1,\dots,k_n}(A)$; $\mathscr{O}_{n}(\mathfrak{a})=|\mathscr{O}_{n}(A)|$.
		%	\item $\mathscr{O}(x)=\bigcup_{_n\in\omega}\mathscr{O}_{n}$; $\mathscr{O}(\mathfrak{a})=|\mathscr{O}(x)|$.
	\end{enumerate}
\end{definition}

%We should note that for each ordered partition $\bar{p}=\langle p_1,\dots,p_n\rangle$, there is an unique partition $p=\ran(\bar{p})\setminus\emptyset\in \mathscr{B}(A)$. 

\begin{fact}\label{bf}
	For all infinite cardinals $\mathfrak{a}$ and all non-zero natural numbers $n$, $\mathscr{B}_n(\mathfrak{a})\leqslant^{\ast}\mathscr{O}_{n}(\mathfrak{a})$.
\end{fact}
\begin{proof}
	For any infinite set $A$, the function that maps each $\bar{p}\in\mathscr{O}_{n}(A)$ to $(\ran(\bar{p})\setminus\{\varnothing\})\cup[A\setminus\bigcup\ran(\bar{p})]^1$ is a surjection from $\mathscr{O}_{n}(A)$ onto $\mathscr{B}_n(A)$.
\end{proof}

\begin{fact}\label{finfin}
	For all infinite cardinals $\fra$, $\mathscr{B}_{\fin}(\fra)\leqslant \fin(\fin(\fra))$.
\end{fact}
\begin{proof}
	For any infinite set $A$, the function that maps each $P\in\mathscr{B}_{\fin}(A)$ to $\mathrm{ns}(P)$ is an injection from $\mathscr{B}_{\fin}(A)$ into $\fin(\fin(A))$. 
\end{proof}

\begin{lemma}\label{shen01}
	For all infinite cardinals $\fra$,
	\[
	\seq(\fra)=^{\ast}\fin(\fin(\fra))=^{\ast}\mathrm{seq}(\fra).
	\]
\end{lemma}
\begin{proof}
	See \cite[Proposition~2.7]{Shen2021}
\end{proof}

\begin{lemma}\label{fin}
	For all infinite cardinals $\mathfrak{a}$ and all non-zero natural numbers $n$, $\fin(\mathfrak{a})^n=\mathscr{O}_{2^n-1}(\mathfrak{a})$.
\end{lemma}
\begin{proof}
	Let $A$ be an arbitrary infinite set and let $n$ be a non-zero natural number. Let $h$ be a bijection between $\{1,\dots,2^n-1\}$ and $\mathscr{P}(n)\setminus\{\varnothing\}$. 
	We define two functions $f:\fin(A)^n\to\mathscr{O}_{2^n-1}(A)$ and $g:\mathscr{O}_{2^n-1}(A)\to\fin(A)^n$ as follows. 
	For all $s\in\fin(A)^n$, $f(s)=\langle p_1,\dots,p_{2^n-1}\rangle$, where
	\[
	p_i=\bigcap_{k\in h(i)}s(k)\setminus\bigcup_{k\in n\setminus h(i)}s(k).
	\]  
	For all $\bar{q}=\langle q_1,\dots,q_{2^n-1}\rangle\in\mathscr{O}_{2^n-1}(A)$, $g(\bar{q})$ is the function on $n$ given by
	\[
	g(\bar{q})(k)=
	\bigcup\{q_{h^{-1}(a)}\mid k\in a\subseteq n\}.
	\]
	It is straightforward to check that $g(f(s))=s$ and $f(g(\bar{q}))=\bar{q}$, so both $f$ and $g$ are bijections.
\end{proof}

\begin{lemma}\label{bfin}
	For all infinite cardinals $\mathfrak{a}$ and all non-zero natural numbers $n$, $\mathscr{B}_{2^n-1}(\mathfrak{a})\leqslant^{\ast}[\fin(\mathfrak{a})]^n$.
\end{lemma}
\begin{proof}
	Let $A$ be an arbitrary infinite set and let $n$ be a non-zero natural number. Let $h$ be a bijection between $\{1,\dots,2^n-1\}$ and $\mathscr{P}(n)\setminus\{\varnothing\}$. 
	For each $s\in[\fin(A)]^n$, let $\sim_s$ be the equivalence relation on $A$ defined by
	\[
	x\sim_s y \qquad \text{if and only if}\qquad \forall p\in s(x\in p\leftrightarrow y\in p). 
	\]
	We define a surjection $f$ from a subset of $[\fin(A)]^n$ onto $\mathscr{B}_{2^n-1}(A)$ as follows. 
	For all $s\in[\fin(A)]^n$,
	if the quotient set $A/{\sim_s}$ has no singleton blocks and $|A/{\sim_s}|=2^n$, then 
	\[
	f(s)=((A/{\sim_s})\cap\fin(A))\cup[\bigcup((A/{\sim_s})\setminus\fin(A))]^1;
	\]
	otherwise, $f(s)$ is undefined. 
	Such a function $f$ is clearly as required.
\end{proof}

\section{A generalization of L{\"a}uchli's lemma}
In this section, we prove a generalization of L{\"a}uchli's lemma, which states that for all infinite cardinals $\mathfrak{a}$ and all non-zero natural numbers $n$, 
\[
(2^{ \mathscr{O}_{n} (\fra)})^{\aleph_0}=2^{\mathscr{B}_{n} (\fra)}.
\]
Note that when $n=1$ we obtain L{\"a}uchli's lemma. Also, using Lemmas~\ref{fin} and~\ref{bfin}, we obtain \eqref{ss01} as follows:
\[
2^{\fin(\mathfrak{a})^n}=2^{ \mathscr{O}_{2^n-1} (\fra)}\leqslant2^{ \mathscr{B}_{2^n-1} (\fra)}\leqslant2^{[\fin(\mathfrak{a})]^n}\leqslant	2^{\fin(\mathfrak{a})^n}.
\]

Fix an arbitrary infinite set $A$ and a non-zero natural number $n$. For a finite sequence $\langle x_1,\dots,x_n\rangle$ of length $n$, we write $\bar{x}=\langle x_1,\dots,x_n\rangle$ for short. For finite sequences $\bar{x}=\langle x_1,\dots,x_n\rangle$ and $\bar{y}=\langle y_1,\dots,y_n\rangle$, we introduce the following abbreviations: $\bar{x}\sqsubseteq\bar{y}$ means that $x_i\subseteq y_i$ for all $i=1,\dots,n$; $\bar{x}\sqsubset\bar{y}$ means that $\bar{x}\sqsubseteq\bar{y}$ and $\bar{x}\not=\bar{y}$; $\bar{x}\sqcup\bar{y}$ is the finite sequence $\langle x_1\cup y_1,\dots,x_n\cup y_n\rangle$. For an operator $\delta$ and $k\in\omega$, we write $\delta^{(k)}(X)$ for $\delta(\delta(\cdots\delta(X)\cdots))$ ($k$~times), and if $k=0$ then $\delta^{(0)}(X)$ is $X$ itself.

\begin{definition}
	For all natural numbers $m_1,\dots,m_n$ and $l_1,\dots.l_n$ such that $m_i\leqslant l_i$ for all $i=1,\dots,n$, we present the following functions: 
	\begin{enumerate}
		\item  $\gamma_{n,\bar{m},\bar{l}}$ is the function on $\mathscr{P}(\mathscr{O}_{m_1,\dots,m_n}(A))$ defined by   
		\[
		\gamma_{n,\bar{m},\bar{l}} (X)=\{\bar{q}\in \mathscr{O}_{l_1,\dots,l_n}(A)\mid \bar{p}\sqsubseteq\bar{q} \text{ for some }\bar{p}\in X\}.
		\]
		\item $\alpha_{n,\bar{m},\bar{l}}$ is the function on $\mathscr{P}( \mathscr{O}_{m_1,\dots,m_n}(A))$ defined by   
		\[
		\alpha_{n,\bar{m},\bar{l}}(X)=\{{\bar{p}}\in \mathscr{O}_{m_1,\dots,m_n}(A)\mid
		\text{ for all }\bar{q}\in \mathscr{O}_{l_1,\dots,l_n}(A)
		\text{, if }\bar{p}\sqsubseteq\bar{q} \text{ then }\bar{q}\in \gamma_{n,\bar{m},\bar{l}} (X)\}.
		\] 
		\item $\delta_{n,\bar{m},\bar{l}}$ is the function on $\mathscr{P}( \mathscr{O}_{m_1,\dots,m_n}(A))$ defined by 
		\[
		\delta_{n,\bar{m},\bar{l}}(X)= \alpha_{n,\bar{m},\bar{l}} (X)\setminus X.
		\]
	\end{enumerate}
\end{definition}

\begin{fact}\label{00}
	For all natural numbers $m_1,\dots,m_n$ and $l_1,\dots.l_n$ such that $m_i\leqslant l_i$ for all $i=1,\dots,n$, we have: 
	\begin{enumerate}
		\item If $X\subseteq Y\subseteq \mathscr{O}_{m_1,\dots,m_n}(A)$, then $\gamma_{n,\bar{m},\bar{l}}(X)\subseteq\gamma_{n,\bar{m},\bar{l}}(Y)$.\label{01}
		\item If $X\subseteq \mathscr{O}_{m_1,\dots,m_n}(A)$, then $X\subseteq\alpha_{n,\bar{m},\bar{l}}(X)$.\label{02}
		\item If $X\subseteq Y\subseteq \mathscr{O}_{m_1,\dots,m_n}(A)$, then $\alpha_{n,\bar{m},\bar{l}}(X)\subseteq\alpha_{n,\bar{m},\bar{l}}(Y)$.\label{03}
		\item If $X\subseteq \mathscr{O}_{m_1,\dots,m_n}(A)$, then $\gamma_{n,\bar{m},\bar{l}}(\alpha_{n,\bar{m},\bar{l}}(X))=\gamma_{n,\bar{m},\bar{l}}(X)$.\label{05}
		\item If $X\subseteq \mathscr{O}_{m_1,\dots,m_n}(A)$, then $\alpha_{n,\bar{m},\bar{l}}(\alpha_{n,\bar{m},\bar{l}}(X))=\alpha_{n,\bar{m},\bar{l}}(X)$.\label{04}
		\item $\gamma_{n,\bar{m},\bar{l}}$ is injective on $\{X\subseteq \mathscr{O}_{m_1,\dots,m_n}(A)\mid \alpha_{n,\bar{m},\bar{l}}(X)=X\}$.\label{06}
		\item Let $l^{\prime}_1,\dots,l^{\prime}_n$ be natural numbers such that $l_i\leqslant l^{\prime}_i$ for all $i=1,\dots,n$. If $X\subseteq \mathscr{O}_{m_1,\dots,m_n}(A)$, then $\alpha_{n,\bar{m},\bar{l}}(X)\subseteq\alpha_{n,\bar{m},\bar{l'}}(X)$, and hence $\alpha_{n,\bar{m},\bar{l'}}(X)=X$ implies that $\alpha_{n,\bar{m},\bar{l}}(X)=X$.\label{07}
		\item If $X\subseteq \mathscr{O}_{m_1,\dots,m_n}(A)$ and $k\in\omega$, then 
		\[
		\delta_{n,\bar{m},\bar{l}}^{(k)}(X)=\alpha_{n,\bar{m},\bar{l}}(\delta_{n,\bar{m},\bar{l}}^{(k)}(X))\setminus \delta_{n,\bar{m},\bar{l}}^{(k+1)}(X).
		\]\label{08}
	\end{enumerate}
\end{fact}

The proof of the above is straightforward and is therefore omitted. We also need the following  version of Ramsey's theorem, whose proof can be found in \cite[Theorem~4.7]{Halbeisen}.

\begin{lemma}\label{prt}
	There exists a function $R_n$ on $\omega^n\times(\omega\setminus\{0\})\times\omega$ such that, for all  $j_1,\dots,j_n,c,r\in \omega$ with $c\neq0$ and all finite sets $S_1,\dots,S_n,Y_1,\dots, Y_c$, if $|S_i|\geqslant R_n(j_1,\dots,j_n,c,r)$ for each $i=1,\dots,n$ and 
	\[
	[S_1]^{j_1}\times\cdots\times[S_n]^{j_n}=Y_1\cup\dots\cup Y_c,
	\]
	then for each $i=1,\dots,n$ there exists a $T_i\in[S_i]^r$ such that 
	\[
	[T_1]^{j_1}\times\cdots\times[T_n]^{j_n}\subseteq Y_d
	\]
	for some $d=1,\dots,c$.
\end{lemma}

The key step of our proof is the following lemma, which essentially uses Lemma~\ref{prt}.

\begin{lemma}\label{hyf07}
	For all natural numbers $m_1,\dots,m_n$ and $l_1,\dots,l_n$ such that $m_i\leqslant l_i$ for all $i=1,\dots,n$, if $X\subseteq \mathscr{O}_{m_1,\dots,m_n}(A)$, then 
	\[
	\delta_{n,\bar{m},\bar{l}}^{(m_1+\dots+m_n+1)}(X)=\varnothing.
	\]
\end{lemma}
\begin{proof}
	In this proof, we omit the subscripts in $\gamma_{n,\bar{m},\bar{l}},\alpha_{n,\bar{m},\bar{l}}$ and $\delta_{n,\bar{m},\bar{l}}$ for the sake of simplicity. 
	
	Let $\phi(X,\bar{p},\bar{q})$ abbreviate the following statement: $X\subseteq \mathscr{O}_{m_1,\dots,m_n}(A)$, $\bar{p},\bar{q}\in \mathscr{O}_{n}(A)$, $|p_i|\leqslant m_i$ for all $i=1,\dots,n$, $\bigcup\ran(\bar{p})\cap\bigcup\ran(\bar{q})=\varnothing$, and $\bar{p}\sqcup\bar{s}\in X$ for all $\bar{s}\in [q_1]^{m_1-|p_1|}\times\dots\times[q_n]^{m_n-|p_n|}$. 
	
	Let $\psi(X,\bar{p})$ abbreviate the following statement: for all $r\in\omega$ there exists a $\bar{q}\in \mathscr{O}_{r,\dots,r}(A)$ such that $\phi(X,\bar{p},\bar{q})$.
	
	We claim that for all $X\subseteq \mathscr{O}_{m_1,\dots,m_n}(A)$ and all $\bar{p}\in \mathscr{O}_{n}(A)$,
	\begin{equation}
		\psi(\delta(X),\bar{p})\text{ implies that }\psi(X,\bar{u})\text{ for some }\bar{u}\sqsubset\bar{p}.\label{C}
	\end{equation}
	If (\ref{C}) is proved, we can prove the lemma as follows. Assume to the contrary that $X\subseteq \mathscr{O}_{m_1,\dots,m_n}(A)$ and there is a $\bar{p}\in \delta^{(m_1+\dots+m_n+1)}(X)$. Clearly, $\psi(\delta^{(m_1+\dots+m_n+1)}(X),\bar{p})$. Repeatedly using (\ref{C}), we get a decreasing sequence 
	\[
	\bar{p}\sqsupset\bar{u}_1\sqsupset\dots\sqsupset\bar{u}_{m_1+\dots+m_n+1},
	\]
	contradicting $\bar{p}\in \mathscr{O}_{m_1,\dots,m_n}(A)\subseteq[A]^{m_1}\times\cdots\times[A]^{m_n}$.
	
	Now, let us prove (\ref{C}). Let $X\subseteq \mathscr{O}_{m_1,\dots,m_n}(A)$ and let $\bar{p}\in \mathscr{O}_{n}(A)$ be such that $\psi(\delta(X),\bar{p})$. It suffices to prove that 
	\[
	\forall r\geqslant l_1+\dots+l_n\,\exists\bar{u}\sqsubset\bar{p}\,\exists\bar{q}\in \mathscr{O}_{r,\dots,r}(A)\,\phi(X,\bar{u},\bar{q}),
	\]
	since then there will be a  $\bar{u}\sqsubset\bar{p}$ such that for infinitely many $r\in\omega$ there is a $\bar{q}\in \mathscr{O}_{r,\dots,r}(A)$ such that $\phi(X,\bar{u},\bar{q})$, and clearly we have $\psi(X,\bar{u})$.
	
	Let $r\geqslant l_1+\dots+l_n$. Let $R_n$ be the function as in Lemma~\ref{prt}. Define
	\[
	r'=\max\{R_n(j_1,\dots,j_n,2,r)\mid j_i\leqslant m_i\text{ for all }i=1,\dots,n\}
	\]
	and
	\[
	r''=R_n(l_1-|p_1|,\dots,l_n-|p_n|,2^{|p_1|+\dots+|p_n|},r').
	\]
	Since $\psi(\delta(X),\bar{p})$, there is an $\bar{S}\in \mathscr{O}_{r'',\dots,r''}(A)$ such that $\phi(\delta(X),\bar{p},\bar{S})$. For all $\bar{u}\sqsubseteq\bar{p}$, let 
	\[
	Y_{\bar{u}}=\{\bar{w}\in [S_1]^{l_1-|p_1|}\times\dots\times[S_n]^{l_n-|p_n|}\mid \bar{u}\sqcup\bar{v}\in X\text{ for some }\bar{v}\sqsubseteq\bar{w}\}.
	\]
	
	We claim that 
	\[
	[S_1]^{l_1-|p_1|}\times\dots\times[S_n]^{l_n-|p_n|}=\bigcup_{\bar{u}\sqsubseteq\bar{p}}Y_{\bar{u}}.
	\]
	Suppose $\bar{w}\in [S_1]^{l_1-|p_1|}\times\dots\times[S_n]^{l_n-|p_n|}$. Take an $\bar{s}\in [S_1]^{m_1-|p_1|}\times\dots\times[S_n]^{m_n-|p_n|}$ such that $\bar{s}\sqsubseteq\bar{w}$. Thus, by $\phi(\delta(X),\bar{p},\bar{S})$, we have  $\bar{p}\sqcup\bar{s}\in \delta(X)\subseteq\alpha(X)$. Since $\bar{p}\sqcup\bar{s}\sqsubseteq\bar{p}\sqcup\bar{w}\in \mathscr{O}_{l_1,\dots,l_n}(A)$, by the definition of  $\alpha$, we have $\bar{p}\sqcup\bar{w}\in\gamma(X)$, and thus $\bar{t}\sqsubseteq\bar{p}\sqcup\bar{w}$ for some $\bar{t}\in X$. Take $\bar{u}=\bar{t}\sqcap\bar{p}$ and $\bar{v}=\bar{t}\sqcap\bar{w}$. Then we have $\bar{u}\sqcup\bar{v}=\bar{t}\in X$, which means that $\bar{w}\in Y_{\bar{u}}$.
	
	Since $|\{\bar{u}\mid \bar{u}\sqsubseteq\bar{p}\}|=2^{|p_1|+\dots+|p_n|}$, by Lemma~\ref{prt}, for each $i=1,\dots,n$, there is a  $T_i\in[S_i]^{r'}$ such that    
	\[
	[T_1]^{l_1-|p_1|}\times\dots\times[T_n]^{l_n-|p_n|}\subseteq Y_{\bar{u}}
	\]
	for some $\bar{u}=\langle u_1,\dots,u_n\rangle\sqsubseteq\bar{p}$.
	Clearly, $\langle T_1,\dots,T_n\rangle\in \mathscr{O}_{r',\dots,r'}(A)$. Let
	\[
	Z=\{\bar{v}\in[T_1]^{m_1-|u_1|}\times\dots\times[T_n]^{m_n-|u_n|}\mid \bar{u}\sqcup\bar{v}\in X\}.
	\]
	Since $|T_i|=r'\geqslant  R_n(m_1-|u_1|,\dots,m_n-|u_n|,2,r)$ for all $i=1,\dots,n$, it follows from Lemma~\ref{prt} that there is some $\bar{q}=\langle q_1,\dots,q_n\rangle\in[T_1]^{r}\times\dots\times[T_n]^{r}\subseteq \mathscr{O}_{r,\dots,r}(A)$ such that either
	\begin{equation}\label{hs55}
		[q_1]^{m_1-|u_1|}\times\dots\times[q_n]^{m_n-|u_n|}\subseteq Z 
	\end{equation}
	or
	\begin{equation}\label{hs66}
		[q_1]^{m_1-|u_1|}\times\dots\times[q_n]^{m_n-|u_n|}\cap Z=\varnothing.
	\end{equation}
	Since $|q_i|=r \geqslant l_i\geqslant l_i-|p_i|$ for all $i=1,\dots,n$, there is a $\bar{w}\in [q_1]^{l_1-|p_1|}\times\dots\times[q_n]^{l_n-|p_n|}$. Since $q_i\subseteq T_i$ for all $i=1,\dots,n$ and $ [T_1]^{l_1-|p_1|}\times\dots\times[T_n]^{l_n-|p_n|}\subseteq Y_{\bar{u}}$, we know $\bar{w}\in Y_{\bar{u}}$. By the definition of $Y_{\bar{u}}$, there is a $\bar{v}\sqsubseteq\bar{w}$ such that $\bar{u}\sqcup\bar{v}\in X$, which excludes \eqref{hs66}. Thus, \eqref{hs55} holds, which means $\phi(X,\bar{u},\bar{q})$.
	
	The last step is to prove that $\bar{u}\neq\bar{p}$. Assume towards a contradiction that $\bar{u}=\bar{p}$. Since $\phi(\delta(X),\bar{p},\bar{S})$ and $\bar{q}\sqsubseteq\bar{S}$, we have $\phi(\delta(X),\bar{p},\bar{q})$. Since also $\phi(X,\bar{u},\bar{q})$, we have $\phi(X,\bar{p},\bar{q})$. Since $|q_i|=r\geqslant l_i\geqslant m_i\geqslant m_i-|p_i|$ for all $i=1,\dots,n$, there is an  $\bar{s}\in [q_1]^{m_1-|p_1|}\times\dots\times[q_n]^{m_n-|p_n|}$. For such an $\bar{s}$, we have both  $\bar{p}\sqcup\bar{s}\in X$ and $\bar{p}\sqcup\bar{s}\in \delta(X)$, contradicting the definition of $\delta$. 
\end{proof}

Now we are ready to prove our main theorem.

\begin{theorem}\label{hyf08}
	For all infinite cardinals $\mathfrak{a}$ and all non-zero natural numbers $n$,
	\[
	(2^{ \mathscr{O}_{n}  (\fra)})^{\aleph_0}=2^{\mathscr{B}_{n}(\mathfrak{a})}.
	\]
\end{theorem}
\begin{proof}
	Let $A$ be an infinite set with $|A|=\mathfrak{a}$. For all natural numbers $j,m_1,\dots,m_n,k$, let $f(j,\bar{m},k)$ be the ordered $n$-tuple
	\[ 
	\langle \mathsf{p}_{0}^i\mathsf{p}_1^j\mathsf{p}_2^{m_1}\cdots \mathsf{p}_{n+1}^{m_n}\mathsf{p}_{n+2}^k\rangle_{1\leqslant i\leqslant n},
	\] 
	in which $\mathsf{p}_l$ is the $l$-th prime number,
	and let $g(j,\bar{m})=f(j,\bar{m},m_1+\dots+m_n)$.
	
	For all $X\in\mathscr{P}(\mathscr{O}_{n}(A))^{\omega}$ and all natural numbers $j,m_1,\dots,m_n,k$, define 
	\begin{align*}
		X_{j,\bar{m}}=&X(j)\cap\mathscr{O}_{m_1,\dots,m_n}(A),\\
		Y_{j,\bar{m},k}=&\alpha_{n,\bar{m},g(j,\bar{m})}(\delta_{n,\bar{m},g(j,\bar{m})}^{(k)}(X_{j,\bar{m}})),\\
		Z_{j,\bar{m},k}=&\gamma_{n,\bar{m},f(j,\bar{m},k)}(Y_{j,\bar{m},k}).
	\end{align*}
	Let $H$ be the function on $\mathscr{P}(\mathscr{O}_{n}(A))^{\omega}$ defined by 
	\[
	H(X)=\{\ran(\bar{q})\cup[A\setminus\bigcup\ran(\bar{q})]^1\mid \exists j,m_1,\dots,m_n,k\in\omega(k\leqslant m_1+\dots+m_n\text{ and }\bar{q}\in Z_{j,\bar{m},k})\}.
	\] 
	We claim that $H$ is an injection from $\mathscr{P}(\mathscr{O}_{n}(A))^{\omega}$ into $\mathscr{P}(\mathscr{B}_{n}(A))$.
	
	Let $X\in \mathscr{P}(\mathscr{O}_{n}(A))^{\omega}$. For all $\bar{q}=\langle q_1,\dots,q_n\rangle\in Z_{j,\bar{m},k}$ and for all $i=1,\dots,n$, $|q_i|\geqslant 2$ since it is the $i$-th term of $f(l,\bar{m},k)$, which implies  $\ran(\bar{q})\cup[A\setminus\bigcup\ran(\bar{p})]^1\in \mathscr{B}_{n}(A)$. Hence $H(X)\subseteq \mathscr{B}_{n}(A)$. We show that $X$ is uniquely determined by $H(X)$ as follows.
	
	For all natural numbers $j,m_1,\dots,m_n,k$ such that $k\leqslant m_1+\dots+m_n$, $Z_{j,\bar{m},k}$ is uniquely determined by $H(X)$, since 
	\begin{equation}\label{hs11}
		Z_{j,\bar{m},k}=\{\bar{q}\in \mathscr{O}_{l_1,\dots,l_n}(A)\mid \ran(\bar{q})\cup[A\setminus\bigcup\ran(\bar{q})]^1\in H(X)\}, 
	\end{equation}
	where $l_i$ is the $i$-th term of $f(j,\bar{m},k)$ for all $i=1,\dots,n$. The point of \eqref{hs11} is that,  since $l_1<\dots<l_n$, the order of $\bar{q}$ can be recovered from the partition $\ran(\bar{q})\cup[A\setminus\bigcup\ran(\bar{q})]^1$ by comparing the cardinalities of its non-singleton blocks.
	
	For all natural numbers $j,m_1,\dots,m_n,k$ such that $k\leqslant m_1+\dots+m_n$, by (5), (6), (7) of Fact~\ref{00}, we know that $Y_{j,\bar{m},k}$ is the unique subset of $\mathscr{O}_{m_1,\dots,m_n}(A)$ such that $\alpha_{n,\bar{m},g(j,\bar{m})}(Y_{j,\bar{m},k})=Y_{j,\bar{m},k}$ and $\gamma_{n,\bar{m},f(j,\bar{m},k)}(Y_{j,\bar{m},k})=Z_{j,\bar{m},k}$, and since $Z_{j,\bar{m},k}$ is uniquely determined by $H(X)$, it follows that $Y_{j,\bar{m},k}$ is also uniquely determined by $H(X)$. 
	
	For all natural numbers $j,m_1,\dots,m_n,k$, by (8) of Fact~\ref{00} and Lemma~\ref{hyf07}, we know that 
	\[
	X_{j,\bar{m}}=Y_{j,\bar{m},0}\setminus(Y_{j,\bar{m},1}\setminus(\cdots(Y_{j,\bar{m},m_1+\dots+m_{n}-1}\setminus Y_{j,\bar{m},m_1+\dots+m_n})\cdots)).
	\]
	Thus, $X_{j,\bar{m}}$ is also uniquely determined by $H(X)$. 
	
	Finally, since 
	\[
	X=\langle\bigcup_{m_1,\dots,m_n\in\omega} X_{j,\bar{m}}\rangle_{j\in\omega},
	\]
	we know that $X$ is uniquely determined by $H(X)$.
	
	It follows that $(2^{ \mathscr{O}_{n} (\fra)})^{\aleph_0}\leqslant2^{\mathscr{B}_{n}(\fra)}$, which implies that $(2^{ \mathscr{O}_{n} (\fra)})^{\aleph_0}=2^{\mathscr{B}_{n}(\fra)}$ by Fact~\ref{bf} and the Cantor--Bernstein theorem.
\end{proof}

\begin{corollary}\label{omega}
	For all infinite cardinals $\mathfrak{a}$ and all non-zero natural numbers $n$,
	\[
	(2^{\mathscr{B}_{n}(\mathfrak{a})})^{\aleph_0}=2^{\mathscr{B}_{n}(\mathfrak{a})}.
	\]
\end{corollary}
\begin{proof}
	Immediately follows from Fact~\ref{bf} and Theorem~\ref{hyf08}. 
\end{proof}

\begin{corollary}\label{finb}
	For all infinite cardinals $\mathfrak{a}$ and all non-zero natural numbers $n$,
	\[
	2^{\fin(\fra)^n}=2^{\mathscr{B}_{2^n-1}(\fra)}.
	\]
\end{corollary}
\begin{proof}
	Immediately follows from Fact~\ref{bf}, Lemma~\ref{fin} and Theorem~\ref{hyf08}. 
\end{proof}

\begin{corollary}\label{finfinfin}
	For all infinite cardinals $\fra$, 
	\[
	2^{\fin(\fin(\fra))}=	2^{\mathscr{B}_{\fin}(\fra)}
	\]
\end{corollary}
\begin{proof}
	Let $A$ be an arbitrary infinite set. For every $n\in\omega$, by the proof of Theorem~\ref{hyf08}, we can define an injection $f_n$ from $\mathscr{P}(\mathscr{O}_{n}(A))$ into $\mathscr{P}(\mathscr{B}_{n}(A))$, and by the proof of  Lemma~\ref{fin}, we can define a bijection $g_n$ between $\fin(A)^n$ and $\mathscr{O}_{2^n-1}(A)$.
	Let $h$ be the function on $\mathscr{P}(\mathrm{seq}(\fin(A)))$ defined by 
	\[
	h(X)=\bigcup_{n\in\omega}f_{2^n-1}(g_n[X\cap\fin(A)^n]).
	\]
	It is easy to see $h$ is an injection from $\mathscr{P}(\mathrm{seq}(\fin(A)))$ into $\mathscr{P}(\mathscr{B}_{\fin}(A))$. Hence, by Fact~\ref{finfin} and Lemma~\ref{shen01}, 
	\[
	2^{\fin(\fin(\fra))}=2^{\mathrm{seq}(\fra)}\leqslant 2^{\mathrm{seq}(\fin(\fra))}\leqslant 2^{\mathscr{B}_{\fin}(\fra)}\leqslant 2^{\fin(\fin(\fra))}.\qedhere
	\]
\end{proof}

In the next section, we will show that Corollaries~\ref{omega}, ~\ref{finb} and~\ref{finfinfin} are optional in the sense that the stronger statements $\mathscr{B}_{n}(\mathfrak{a})\cdot\aleph_0\leqslant^{\ast}\mathscr{B}_{n}(\mathfrak{a})$, $\fin(\fra)^n\leqslant^{\ast}\mathscr{B}_{2^n-1}(\fra)$ and $\fin(\fin(\fra))\leqslant^{\ast}	\mathscr{B}_{\fin}(\fra)$
cannot be proved in $\mathsf{ZF}$.
\section{Consistency results }
In this section, we prove several consistency results using the method of permutation models. We refer the reader to  \cite[Chapter~8]{Halbeisen} for an introduction to the theory of permutation models. Although permutation models are only models of $\mathsf{ZFA}$ (the Zermelo--Fraenkel set theory with atoms), the Jech--Sochor theorem yields the desired $\mathsf{ZF}$-models.  

Here we only consider the basic Fraenkel model $\mathcal{V}_{\mathrm{F}}$ (see~\cite[pp.~195--196]{Halbeisen}). The set $A$ of atoms of $\mathcal{V}_{\mathrm{F}}$ is denumerable, and $x\in\mathcal{V}_{\mathrm{F}}$ if and only if $x\subseteq\mathcal{V}_{\mathrm{F}}$ and $x$ has a finite support, that is, a set $E\in\fin(A)$ such that every permutation of $A$ fixing $E$ pointwise also fixes $x$.

\begin{lemma}\label{fb}
	Let $A$ be the set of atoms of $\mathcal{V}_{\mathrm{F}}$ and let $\mathfrak{a}=|A|$. In $\mathcal{V}_{\mathrm{F}}$, for all non-zero natural numbers $n$,
	\[
	2^{\mathscr{O}_n(\mathfrak{a})}<	2^{\mathscr{O}_{n+1}(\mathfrak{a})}.
	\]
\end{lemma}
\begin{proof}
	Let $n$ be a non-zero natural number. We claim that, in $\mathcal{V}_{\mathrm{F}}$,
	\[
	2^{\mathfrak{a}^{\underline{n+1}}}\nleqslant2^{\mathscr{O}_n(\mathfrak{a})}.
	\]
	Assume to the contrary that there exists an injection $f\in V_F$ from $\mathscr{P}(A^{\underline{n+1}})$ into $\mathscr{P}(\mathscr{O}_n(A))$. Let $E$ be a finite support of $f$. Take an arbitrary $B\in[A\setminus E]^{n+2}$ and an $s\in B^{\underline{n+1}}$.
	
	A permutation $\pi$ of $B$ is said to be even (odd) if $\pi$ can be written as a product of an even (odd) number of transpositions. We shall view permutations of $B$ as permutations of $A$ by understanding that they are extended to be the identity outside of $B$. Now, let 
	\[
	\xi=\{\pi(s)\mid \pi \text{ is an even permutation of }B \},
	\]
	and let 
	\[
	\vartheta=\{\sigma(s)\mid \sigma \text{ is an odd permutation of }B \}.
	\]
	Since a permutation of $B$ cannot be both even and odd, it follows that $\{\xi,\vartheta\}$ is a partition of $B^{\underline{n+1}}$. It is also clear that $\pi(\xi)=\xi$ for all even permutations $\pi$ of $B$ and $\sigma(\xi)=\vartheta$ for all odd permutations $\sigma$ fo $B$. 
	
	Now, let us consider $f(\xi)$. For each $\bar{p}=\langle p_1,\dots,p_n\rangle\in f(\xi)$, let $\sim_{\bar{p}}$ be the equivalence relation on $B$ induced by $\bar{p}$, that is, for all $a,b\in B$, 
	\[
	a\sim_{\bar{p}} b\qquad\text{ if and only if }\qquad \text{either }a,b\in p_i\text{ for some }i=1,\dots,n\text{ or }a,b\in B\setminus\bigcup^n_{i=1}p_i.
	\]
	For all even permutations $\pi$ of $B$, since $E$ is a finite support of $f$ and $E\cap B=\emptyset$, it follows that $\pi(f)=f$, and thus $\pi(f(\xi))=f(\pi(\xi))=f(\xi)$. Let $\sigma$ be an odd permutation of $B$. We claim that 
	\[
	\sigma(f(\xi))=f(\xi)
	\]
	In fact, for all $\bar{p}\in f(\xi)$, since $|B/{\sim_{\bar{p}}}|\leqslant n+1$ and $|B|=n+2$, there are $a,b\in B$ such that $a\not=b$ and $a\sim_{\bar{p}} b$, and thus the transposition $\tau$ that swaps $a$ and $b$ fixes $\bar{p}$, which implies that $\sigma(\bar{p})=(\sigma\circ\tau)(\bar{p})\in f(\xi)$ since $\sigma\circ\tau$ is even. Thus, $\sigma(f(\xi))=f(\xi)$. It follows that $f(\xi)=\sigma(f(\xi))=f(\sigma(\xi))=f(\vartheta)$, contracting that $f$ is an injection. Thus $2^{\mathfrak{a}^{\underline{n+1}}}\nleqslant2^{\mathscr{O}_n(\mathfrak{a})}$.
	
	It is straightforward to see that $\mathfrak{a}^{\underline{n+1}}\leqslant\mathscr{O}_{n+1}(\mathfrak{a})$, and hence $2^{\mathfrak{a}^{\underline{n+1}}}\leqslant 2^{\mathscr{O}_{n+1}(\mathfrak{a})}$, so $2^{\mathscr{O}_n(\mathfrak{a})}<	2^{\mathscr{O}_{n+1}(\mathfrak{a})}$.
\end{proof}

Now the next corollary immediately follows from Lemma~\ref{fb} and the Jech--Sochor theorem.

\begin{corollary}\label{per}
	It is consistent with $\mathsf{ZF}$ that there is an infinite cardinal $\mathfrak{a}$ such that 	for all non-zero natural numbers $n$,
	\[
	2^{\mathscr{O}_n(\mathfrak{a})}<2^{\mathscr{O}_{n+1}(\mathfrak{a})}.
	\]
\end{corollary}

\begin{theorem}
	It is consistent with $\mathsf{ZF}$ that there is an infinite cardinal $\mathfrak{a}$ such that 
	\[
	\begin{matrix}
		2^{\fin(\fra)} &&<&&2^{\fin(\fra)^2}&<&\cdots&<&2^{\fin(\fra)^n}&<&\cdots&<&2^{\fin(\fin(\fra))}\\
		\veq &&&&\veq&&&&\veq&&&&\veq\\
		2^{\mathscr{B}_1(\fra)} &<&2^{\mathscr{B}_2(\fra)}&<&2^{\mathscr{B}_3(\fra)}&<&\cdots&<&2^{\mathscr{B}_{2^n-1}(\fra)}&<&\cdots&<&2^{\mathscr{B}_{\fin}(\fra)}
	\end{matrix}
	\]

\end{theorem}
\begin{proof}
	Immediately follows from Lemma~\ref{fin}, Theorem~\ref{hyf08}, and Corollaries~\ref{finfinfin} and~\ref{per}.
\end{proof}

\begin{lemma}\label{hs20}
	Let $A$ be the set of atoms of $\mathcal{V}_{\mathrm{F}}$.  In $\mathcal{V}_{\mathrm{F}}$, for all non-zero natural numbers $n$, $\mathscr{B}_n(A)$ is dually Dedekind finite, that is, every surjection of $\mathscr{B}_n(A)$ onto $\mathscr{B}_n(A)$ is injective.
	
\end{lemma}
\begin{proof}
	Let $n$ be a non-zero natural number. Take an arbitrary surjection $f\in\mathcal{V}_\mathrm{F}$ from $\mathscr{B}_n(A)$ onto $\mathscr{B}_n(A)$.
	In order to prove the injectivity of $f$, it suffices to show that
	\begin{equation}\label{s024}
		\text{for all $P\in\mathscr{B}_n(A)$ there is an $m>0$ such that $f^{(m)}(P)=P$.}
	\end{equation}
	Let $E$ be a finite support of $f$. For every $P\in\mathscr{B}_n(A)$ and every $p\in P$, let $p_E=p\setminus E$ and let 
	\[
	P_E=\{p_E\mid p\in \mathrm{ns}(P)\text{ and }p_E\neq \varnothing\}.
	\] 
	Clearly, $|P_E|\leqslant n$ and $P_E$ is a partition of a finite subset of $A\setminus E$. We define  a preordering  $\sqsubseteq$ on $\mathscr{B}_n(A)$ by
	\[
	Q\sqsubseteq P \qquad  \text{if and only if} \qquad \text{every block of  } Q_E \text{ is a union of some blocks of } P_E.
	\]
	
	%	\item $Q\prec P$ if and only if $Q\sqsubseteq P$ and there exists a block of $P_E$ that is contained in no block of $Q_E$.
	
	%Clearly, $\sqsubseteq$ is a preordering on $\mathscr{B}_n(A)$ and $\prec$ is a strict partial ordering on $\mathscr{B}_n(A)$ such that $R\sqsubseteq Q\prec P$ or $R\prec Q\sqsubseteq P$ only if $R\prec P$.
	
	We claim that for every $P\in\mathscr{B}_n(A)$,  $f(P)\sqsubseteq P$.
	Assume towards a contradiction that ${f(P)}\not\sqsubseteq{P}$ for some $P\in\mathscr{B}_n(A)$. Then there is a block $q_E$ of  $f(P)_E$ such that either $q_E$ contains a singleton block $\{a\}$ of $P$ or there is a $p_E\in P_E$ such that both $p_E\cap q_E$ and $p_E\setminus q_E$ are non-empty. In the first case, since $q_E$ is finite, there is a $b\in A\setminus (E\cup q_E)$ such that $\{b\}$ is a singleton block of $P$, so the transposition that swaps $a$ and $b$ fixes $P$ but move $f(P)$, contradicting that $E$ is a support of $f$. In the latter case, the transposition that swaps an element of $p_E\cap q_E$ with an element of $p_E\setminus q_E$ fixes $P$ but move $f(P)$, contradicting again that $E$ is a support of $f$.
	
	We claim that for every $P\in\mathscr{B}_n(A)$,  
	\[
	|\{Q\in\mathscr{B}_n(A)\mid Q_E= P_E\}|\leqslant (n+1)^{|E|}.
	\] 
	Let $P\in \mathscr{B}_n(A)$. 
	Define a function $g:\{Q\in\mathscr{B}_n(A)\mid Q_E= P_E\}\to (P_E\cup \{\varnothing\})^{E}$ by 
	$g(Q)(e)= q_E$,   where $q$ is the unique element of $Q$ such that $e\in q$. It is straightforward to see that $g$ is an injection.
	
	For every $P,Q\in \mathscr{B}_n(A)$, if $Q\sqsubseteq P$ and $|Q_E|=|P_E|$, then $Q_E=P_E$. As a consequence, since $|P_E|\leqslant n$, every $\sqsubseteq$-chain without repetition must have length at most $(n+1)^{|E|+1}$.
	
	Finally, we prove $\eqref{s024}$ as follows. 
	Let $h$ be a function from $(n+1)^{|E|+1}+1$ into $\mathscr{B}_n(A)$ such that $h(0)=P$ and $h(i)=f(h(i+1))$ for all $i<(n+1)^{|E|+1}$.
	Such an $h$ exists since $f$ is surjective. Clearly, for all $i\leqslant (n+1)^{|E|+1}$, $f^{(i)}(h(i))=P$.
	Since $h$ is a $\sqsubseteq$-chain, we can find $i,j\leqslant (n+1)^{|E|+1}$ such that $i<j$ and $h(i)=h(j)$. Now, if we take $m=j-i$, then we have $m>0$ and
	\[
	f^{(m)}(P)=f^{(j-i)}(P)=f^{(j-i)}(f^{(i)}(h(i)))=f^{(j)}(h(j))=P,
	\]
	which completes the proof.
\end{proof}

Now the next corollary immediately follows from Lemma~\ref{hs20} and the Jech--Sochor theorem.

\begin{corollary}\label{hs30}
	It is consistent with $\mathsf{ZF}$ that there is an infinite cardinal $\mathfrak{a}$ such that 	for all non-zero natural numbers $n$,	$\mathscr{B}_n(A)$ is dually Dedekind finite.
\end{corollary}

\begin{corollary}
	It is consistent with $\mathsf{ZF}$ that there is an infinite cardinal $\mathfrak{a}$ such that 	for all non-zero natural numbers $n$,
	$\mathscr{B}_{n}(\mathfrak{a})\cdot\aleph_0\nleqslant^{\ast}\mathscr{B}_{n}(\mathfrak{a})$ and $\fin(\fra)^n\nleqslant^{\ast}\mathscr{B}_{2^n-1}(\fra)$.
\end{corollary}
\begin{proof}
	Note that, by Lemma~\ref{bfin}, for all infinite sets $A$ and all non-zero natural numbers $n$, there are non-injective surjections from $\mathscr{B}_{n}(A)\times\omega$ onto $\mathscr{B}_{n}(A)$ and from $\fin(A)^n$ onto $\mathscr{B}_{2^n-1}(A)$. Hence this corollary follows from Corollary~\ref{hs30}. 
\end{proof}

\begin{lemma}\label{seq}
	Let $A$ be the set of atoms of $\mathcal{V}_{\mathrm{F}}$, and let $\mathfrak{a}=|A|$. In $\mathcal{V}_{\mathrm{F}}$,
	\[
	\fin(\fin(\fra))\not\leqslant^\ast\mathscr{B}_{\fin}(\fra).
	\]
\end{lemma}
\begin{proof}
	By Lemma~\ref{shen01}, it suffices to show $\seq(\fra)\not\leqslant^\ast\mathscr{B}_{\fin}(\fra)$. Assume to the contrary that there exists a surjection $f\in \mathcal{V}_{\mathrm{F}}$ from $\mathscr{B}_{\fin}(A)$ onto $\mathrm{seq}^{1-1}(A)$. Let $E$ be a finite support of $f$. Let $n=|E|$. Take an $s\in (A\setminus E)^{\underline{n+1}}$. Since $f$ is surjective, there is a $P\in \mathscr{B}_{\fin}(A)$  such that $f(P)=s$. Let $\sim_P$ be the equivalence relation on $A$ induced by $P$.
	
	We claim that for every $a\in\ran(s)$, $a\sim b$ for no $b\in A\setminus (E\cup\{a\})$ and $a\sim_P e$ for some $e\in E$. Suppose not; then one of the following two cases occurs.    The first case is that there is a $b\in  A\setminus (E\cup\{a\})$ such that $a\sim_P b$. Then the transposition that swaps $a$ and $b$ will fix $P$ and move $s$, contradicting that $E$ is a support of $P$.
	Otherwise, $\{a\}\in P$. Since $P$ in finitary, there is a $c\in A\setminus (E\cup\{a\})$ such that $\{c\}\in P$.  Then the transposition that swaps $a$ and $c$ will fix $P$ and move $s$, contradicting again that $E$ is a support of $P$.
	
	Hence, the function that maps each $a\in\ran(s)$ to some $e\in E$ with $a\sim_P e$ is an injection from $\ran(s)$ into $E$, which implies $n+1=|\ran(s)|\leqslant |E|=n$, a contradiction.
\end{proof}

Now the next corollary immediately follows from Lemma~\ref{seq} and the Jech--Sochor theorem.

\begin{corollary}
	It is consistent with $\mathsf{ZF}$ that there is an infinite cardinal $\fra$ such that $\fin(\fin(\fra))\not\leqslant^\ast\mathscr{B}_{\fin}(\fra)$.
\end{corollary}

\providecommand{\WileyBibTextsc}{}
\let\textsc\WileyBibTextsc
\providecommand{\othercit}{}
\providecommand{\jr}[1]{#1}
\providecommand{\etal}{~et~al.}


\begin{thebibliography}{10}
	
	\bibitem{Halbeisen}% article
	\textsc{L. Halbeisen},
	\emph{Combinatorial Set Theory: With a Gentle Introduction to Forcing}, 2nd ed.
	(Springer, 2017).
	
	\bibitem{Lauchli1961}% article
	\textsc{H.~L\"auchli},
	\emph{Ein Beitrag zur Kardinalzahlarithmetik ohne Auswahlaxiom},
	\jr{Math. Log. Q.} \textbf{7}, 141--145 (1961).
	
	\bibitem{Shen2021}% article
	\textsc{G.~Shen},
	\emph{A choice-free cardinal equality},
	\jr{Notre Dame J. Form. Log.} \textbf{62}, 577--587 (2021).
	
	\bibitem{Shen2024}% article
	\textsc{G.~Shen},
	\emph{Cantor's theorem mat fail for finitary partitions},
	\jr{J. Symb. Log.}, 1--18 (2024), doi:10.1017/jsl.2024.24.
	
	\bibitem{sv2024}% article
	\textsc{N.~Sonpanow} and \textsc{P.~Vejjajiva} ,
	\emph{A generalisation of Läuchli’s lemma},
	\jr{Math. Log. Q.} \textbf{70}, 173--177 (2024).
	
	
\end{thebibliography}
\end{document}